\documentclass[12pt]{elsarticle}
\usepackage[utf8]{inputenc}
\usepackage[english]{babel}
\usepackage{amssymb,amsmath,amsthm,amsfonts,xcolor,enumerate,hyperref,comment,longtable,cleveref}

\def\ov{\overline}

\theoremstyle{plain}%
  \newtheorem{theorem}{Theorem}[section]
  \newtheorem{corollary}[theorem]{Corollary}
  \newtheorem{proposition}[theorem]{Proposition}
  \newtheorem{lemma}[theorem]{Lemma}

  \newtheorem{definition}[theorem]{Definition}

\newtheorem{remark}[theorem]{Remark}

\newfont{\hueca}{msbm10}

\begin{document}

\noindent{\Large Decompositions of linear spaces induced by $n$-linear maps}\footnotetext{The work is  supported by FAPESP 17/15437-6; by the PCI of the UCA `Teor\'\i a de Lie y Teor\'\i a de Espacios de Banach', by the PAI with project numbers FQM298, FQM7156 and by the project of the Spanish Ministerio de Educaci\'on y Ciencia  MTM2016-76327C31P. }

   \

   {\bf   Antonio Jesús Calder\'on$^{a}$,  Ivan Kaygorodov$^{b}$, Paulo Saraiva$^{c}$ \\

\bigskip


{\bf References}

    \medskip
}

\renewcommand{\thefootnote}{}

{\tiny

$^{a}$ Departamento de Matem\'aticas, Universidad de C\'adiz. Puerto Real, C\'adiz, Espa\~na.

$^{b}$ CMCC, Universidade Federal do ABC. Santo Andr\'e, Brasil.

$^{c}$ Universidade de Coimbra, CeBER, CMUC e FEUC, Coimbra, Portugal
\

\smallskip

   E-mail addresses:

\smallskip   
     Antonio Jesús Calder\'on  (ajesus.calderon@uca.es),

\smallskip
    Ivan Kaygorodov (kaygorodov.ivan@gmail.com),

\smallskip    
Paulo Saraiva (psaraiva@fe.uc.pt).

}

\

\

{\bf ABSTRACT.}
Let $\mathbb V$ be an arbitrary linear space and $f:\mathbb V \times \ldots \times  \mathbb V \to \mathbb V$ an $n$-linear map. It is proved that, for each choice of a basis ${\mathcal B}$ of $\mathbb V$, the $n$-linear map $f$ induces a (nontrivial) decomposition $\mathbb V= \oplus V_j$ as a direct sum of linear subspaces of $\mathbb V$, with respect to ${\mathcal B}$. It is shown that this decomposition is $f$-orthogonal  in the sense that $f(\mathbb V, \ldots, V_j, \ldots, V_k, \ldots, \mathbb  V) =0$ when $j \neq k$, and in such a way that any $V_j$ is strongly $f$-invariant, meaning that  $f(\mathbb  V, \ldots, V_j, \ldots, \mathbb V) \subset V_j.$  
A sufficient condition for two different decompositions of $\mathbb V$ induced by an $n$-linear map $f$, with respect to two  different bases of $\mathbb V$, being isomorphic is deduced. The $f$-simplicity -- an analogue of the usual simplicity in the framework of $n$-liner maps -- of any linear subspace $V_j$ of a certain decomposition induced by $f$ is characterized. Finally,  an application to the structure theory   of arbitrary $n$-ary algebras is provided. This work is a close generalization the results obtained by A. J. Calderón (2018)  \cite{Yo4}.
\bigskip


 \ 
 
\medskip

{\it 2010MSC}: 15A03, 15A21, 15A69, 15A86.

\newpage

\section{Introduction}

The main idea of this paper is to present an $n$-ary ($n>2$) generalization of the results achieved by the first author on the decomposition of linear spaces induced by bilinear maps on a linear space \cite{Yo4}.

In the mentioned paper, given a linear space $\mathbb V$ of arbitray dimension and a bilinear map $f$ on $\mathbb V$, Calderón introduced the notions of $f$-orthogonal, $f$-invariant and strongly $f$-invariant subspaces, as well as the notion of $f$-simplicity, which are just the usual notions of orthogonality, invariance and simplicity, but now defined with respect to $f$. Then, for a fixed basis of $\mathbb V$, he developed connection tecnhiques allowing to obtain a first nontrivial decomposition of $\mathbb V$ as the direct sum of $f$-orthogonal vector subspaces. In order to improve the obtained decomposition he introduced an adequate equivalence relation on the above family of linear subspaces, leading to the first main result: a nontrivial decomposition of $\mathbb V$ as an $f$-orthogonal direct sum of strongly $f$-invariant linear subspaces, with respect to a fixed basis. After that, observing that different choices of the bases of $\mathbb V$ may lead to different decompositons, he studied sufficient conditions to assure induced isomorphic decompositions of $\mathbb V$ with respect to different bases of $\mathbb V$. Another important result gives necessary and sufficient conditions for the $f$-simplicity of  the linear subspaces in the second  decomposition of $\mathbb V$. The author ends the paper providing an application of the previous results to the the structure theory of arbitrary algebras.

At this point, a parenthesis is due to underline the considerable amount of recent works where the above mentioned and similar connection techniques are applied as a tool to obtain interesting results in the frameworks of several types of algebras. Without being exhaustive, these techniques were used, for instance, along with the notions of multiplicative basis and quasi-multiplicative basis not only related with algebras (see Caledrón and Navarro, \cite{Yo,Yo2}), but also with some $n$-ary generalizations (see, {\it e.g.}, the works of Calder\'on, Barreiro, Kaygorodov and S\'anchez in \cite{bcks,kmod,Yo_n_algebras}). Further, connection techniques were also applied in the context of graded Lie algebras (see Calderón (2014) \cite{Yo3}) and to obtain structural results on graded Leibniz triple systems (see Cao and Chen (2016) \cite{Cao2}).

The present work follows an approach that uses, as close as possible, generalized $n$-ary versions of the techniques applied in \cite{Yo4}, obtaining generalized results which are similar to those of Calderón.

The paper is organized as follows. In Section 2 we present the necessary basic notions related with $n$-linear maps and develop all connection techniques needed to obtain the main results. As a consequence, we get that each choice of a basis ${\mathcal B} $ of  $\mathbb V$ rises a first nontrivial decomposition of $\mathbb V$, induced by $f$, as an $f$-orthogonal direct sum of linear subspaces with respect to ${\mathcal B} $. This decomposition is then enhanced by the introduction of an adequate equivalence relation on the above family of linear subspaces, leading to our first main result: $\mathbb V$ decomposes as a nontrivial $f$-orthogonal direct sum of strongly $f$-invariant linear subspaces, with respect to a fixed basis. 

In Section 3 the relation among the previous decompositions of $\mathbb V$ given by different choices of its bases is discussed. Concretely, after defining the notion of orbit associated to an $n$-linear map $f$, it is shown that if two bases, ${\mathcal B} $ and ${\mathcal B}^{\prime} $  of  $\mathbb V$ belong to the same orbit under an action of a certain subgroup of ${\rm GL}(\mathbb V)$ on the set of all bases of  $\mathbb V$, then they induce isomorphic decompositions of  $\mathbb V$.

In Section 4 we generalize the concept of $i$-division basis to the case of $n$-ary algebras. After that, we obtain a characterization of the $f$-simplicity of the components  of the main decomposition obtained in Section 2. That is, we prove that any of the linear subspaces in the decomposition of $\mathbb V$  in $f$-orthogonal, strongly $f$-invariant linear subspaces of $\mathbb V$  is $f$-simple if and only if its annihilator is zero and it admits an $i$-division basis. 

Finally, in Section 5 an application of the previous results to the the structure theory of arbitrary $n$-ary algebras is included.

\section{Development of the techniques. First decomposition theorem}

We begin by noting that throughout the paper all  of the  linear spaces $\mathbb  V$ considered are of arbitrary dimension and over an arbitrary base field ${\mathbb F}$.
Hereinafter, $\mathbb V$ is a linear space and  $f:  \mathbb V\times \dots \times \mathbb V \to \mathbb V$ an $n$-linear map on $\mathbb V$, $n\geq 2$.
We start recalling some notions concerning $\mathbb V$ and $f$.

\begin{definition}\rm
 Two linear subspaces $V_1$ and $V_2$ of $\mathbb V$ are called {\it $f$-orthogonal} if
$$f(\mathbb V,\dots, V_1^{(i)},\dots,V_2^{(j)},\dots,\mathbb V)=0,$$
for all $i,j\in \left\{1,\dots,n \right\}$, $i \neq j$, where the notations $V_1^{(i)}$ and $V_2^{(j)}$ mean that $V_1$ and $V_2$ occupy the $i$-th and $j$-th entries of $f$, respectively.

It is also said that a decomposition of $\mathbb V$ as a direct sum of linear subspaces
$$\mathbb V= \bigoplus_{j\in J} V_j$$ is {\it $f$-orthogonal} if $V_j$ and $V_k$ are $f$-orthogonal for any $j,k \in J$, with $j \neq k
$.
\end{definition}

\begin{definition}\rm
A  linear subspace $ W$ of $\mathbb V$ is  called {\it $f$-invariant } if
$$f(W,\dots,W)\subset W.$$ 
The linear space $W$ is called  {\it strongly $f$-invariant } if 
$$f(\mathbb V,\dots, W^{(i)},\dots,\mathbb V)\subset W,$$
for all $i \in \left\{1,\dots,n \right\}$. The linear space $\mathbb V$ will be called {\it $f$-simple} if 
$$f(\mathbb V,\dots,\mathbb V) \neq 0$$ 
and its only strongly $f$-invariant subspaces are $\{0\}$ and $\mathbb V$.
\end{definition}

\begin{definition}
The {\it annihilator} of  $f$ is defined as the set
$${\rm Ann}(f)=\{v \in \mathbb V: f(\mathbb V,\dots, v^{(i)},\dots,\mathbb V)=0, \text{ for all } i \in \left\{1,\dots,n \right\}\}. $$
\end{definition}

\medskip

Let us fix a basis ${\mathcal B}=\{e_i\}_{ i \in I}$ of $\mathbb V$. For each $e_i \in {\mathcal B}$, we introduce a symbol  $\overline{e}_i \notin {\mathcal B}$ and the following set
$$\overline{{\mathcal B}} := \{\overline e_i : e_i \in {\mathcal B}\}.$$ 
We will also write $\overline{(\overline{e}_i)} := e_i \in {\mathcal B}$, $\mathbb V^{*}:=\mathbb V \setminus \{0\}$ and  $\mathcal{P}(\mathbb V^{*})$ the power set of $\mathbb V^{*}$.
\medskip

We define the $n$-linear mapping

\begin{equation}\label{gene}
  F: \mathcal{P}(\mathbb V^{*})\times \left( ({\mathcal B}\dot\cup \overline{{\mathcal B}})\times \dots  \times ({\mathcal B}\dot\cup \overline{{\mathcal B}} )\right)  \to \mathcal{P}(\mathbb V^{*})
 \end{equation}
   as

\begin{itemize}

\item[(i)]  $ F(\emptyset ,  {\mathcal B}\dot\cup \overline{{\mathcal B}}, \dots, {\mathcal B}\dot\cup \overline{{\mathcal B}} )= \emptyset$.

 \item[(ii)]   For any $\emptyset \neq U \in {\mathcal P}( \mathbb V^{*})$  and $\xi_i \in {\mathcal B} $, $i=1,\dots,n-1,$
$$\hspace{-1cm} F(U, \xi_{1},\dots, \xi_{n-1})=\left( \bigcup_{  
\begin{array}{c}
k\in 	\{1,\dots,n\}\\  
\sigma \in {\mathbb S}_{n-1}
\end{array}
} \{f(\xi_{\sigma (1)},\dots, u^{(k)}, \ldots, \xi_{\sigma (n-1)}):u \in U\} \right)\setminus \{0\} .$$

 \item[(iii)]   For any $\emptyset \neq U \in {\mathcal P}(\mathbb  V^{*})$  and $\ov{\xi}_i \in
 \overline{{\mathcal B}}$, $i=1,\dots,n-1,$
 $$\hspace{-1cm}  F(U ,  \ov{\xi}_{1},\dots,\ov{\xi}_{n-1})=\left(
\bigcup_{  
\begin{array}{c}
k\in 	\{1,\dots,n\}\\  
\sigma \in {\mathbb S}_{n-1}
\end{array}
}  \{u \in \mathbb V:  f(\xi_{\sigma (1)},\dots, u^{(k)},\dots, \xi_{\sigma (n-1)})\in U\} \right)\setminus \{0\} .$$

 \item[(iv)]   $F(U ,  \xi_{1},\dots, \xi_{n-1})=\emptyset$, if there are  $i,j\in \{1,\dots,n-1\},\ i\neq j$, such that $\xi_{i}\in {\mathcal B}$, $\xi_{j}\in \overline{{\mathcal B}}$.
 \end{itemize}

\begin{remark}  \label{Fsym}
It is clear that $$F(U , \xi_{\sigma (1)},\dots, \xi_{\sigma (n-1)}) = F(U ,\xi_{1},\dots,\xi_{n-1}),$$ and $$F(U , \ov{\xi}_{\sigma (1)},\dots, \ov{\xi}_{\sigma (n-1)}) = F(U ,\ov{\xi}_{1},\dots,\ov{\xi}_{n-1}),$$ for all $\xi_{1},\dots,\xi_{n-1}\in 	{\mathcal B},\ \ov{\xi}_{1},\dots,\ov{\xi}_{n-1}\in \ov{{\mathcal B}},\ \sigma \in {\mathbb S}_{n-1}$.
\end{remark}

\begin{lemma}\label{lema1}
Concerning the mapping $F$ previously defined, we have
\begin{enumerate}
\item[1.] For any $v \in \mathbb V^{*}$ and $\xi_i \in {\mathcal B}\ i=1,\dots,n-1$,\\  $w \in F(\{v\} , \xi_{1},\dots,\xi_{n-1})$ if and only if $v \in F(\{w\},\ov{\xi}_{1},\dots,\ov{\xi}_{n-1})$.


\item[2.] For any $U \in \mathcal{P}(\mathbb V^{*})$ and  $\xi_{i}\in {\mathcal B}\dot\cup \overline{{\mathcal B}}, \ i=1,\dots,n-1$,\\ $v \in F(U ,\xi_{1},\dots,\xi_{n-1})$ if and only if 
$F(\{v\},\ov{\xi}_{1},\dots,\ov{\xi}_{n-1})\cap U \neq \emptyset $.
\end{enumerate}

\end{lemma}

\begin{proof}
1. Let us start admitting that $w \in F(\{v\} , \xi_{1},\dots,\xi_{n-1})$, being $v \in  \mathbb  V^{*}$ and $\xi_i \in {\mathcal B},\ i=1,\dots,n-1$. This means that 
$$w=f(\xi_{\sigma (1)},\dots,v^{(k)}, \dots, \xi_{\sigma (n-1)}), $$
for some $k\in \{1, \dots ,n-1\}$ and $\sigma \in {\mathbb S}_{n-1}$, and thus 
$$v\in F(\{w\}, \ov{\xi}_{\sigma (1)},\dots, \ov{\xi}_{\sigma (n-1)}). $$
According to the previous remark, we have:
$$v\in F(\{w\}, \ov{\xi}_{1},\dots, \ov{\xi}_{n-1}). $$
The reciprocal result can be proved analogously.


2. Suppose that $U\in \mathcal{P}(\mathbb V^{*})$ and $\xi_{i}\in {\mathcal B}\dot\cup \overline{{\mathcal B}}, \ i=1,\dots,n-1.$ 
Let us first admit that $v \in F(U ,\xi_{1},\dots,\xi_{n-1})$. Then $v \in F(\{w\},\xi_{1},\dots,\xi_{n-1})$ for some $w\in U$. By item 1., this is equivalent to 
$w\in F(\{v\}, \ov{\xi}_{1},\dots, \ov{\xi}_{n-1})$ and thus 
$$w\in F(\{v\},\ov{\xi}_{1},\dots, \ov{\xi}_{n-1})\cap U \neq \emptyset .$$
The reciprocal assertion can be proved in a similar way.
\end{proof}

\begin{definition}\label{connection}\rm
Let $e_i, e_j \in {\mathcal B}$. We say that $e_i $ is {\it connected} to $e_j$ if either,

\begin{itemize}
\item[(i)] $e_i=e_j$ or

\item[(ii)]  there exists an ordered   list $(X_1,X_2,\dots,X_m)$, where $X_{i}=\left(a_{i1},\dots,a_{in-1} \right)$ such that $a_{ik} \in  {\mathcal B}\dot\cup \overline{{\mathcal B}} $, $i \in \{1,\dots,m\},\ k\in \{1,\dots,n-1\},$  satisfying:

\begin{enumerate}
\item [{\rm 1.}]
$F(\{e_i\} , X_1)\neq\emptyset$,\\
$F(F(\{e_i\} , X_1),X_{2})\neq\emptyset,\\
 \vdots\\
 F (\dots (F(F(\{e_i\} , X_1),X_{2}),\dots,X_{m-1})\neq\emptyset$.

\medskip

\item [{\rm 2.}]
$e_j\in F( F (\dots (F(F(\{e_i\} , X_1),X_{2}),\dots,X_{m-1}), X_{m}).$
\end{enumerate}

\end{itemize}

In this case we say that $(X_1,X_2,\dots,X_m)$ is a {\it connection} from $e_i$ to $e_j$.
\end{definition}

\begin{lemma}\label{lema3}
Let $(X_1, X_2, \dots, X_{m-1}, X_m)$ be any connection from $e_i$ to  $e_j$, where $e_i$ and $e_j$ are arbitrary elements in ${\mathcal B}$, with $e_i \neq e_j$. Then the ordered list  $(\overline{X}_m,\overline{X}_{m-1},\dots,\overline{X}_2,\overline{X}_1)$ is a connection from $e_j$ to $e_i$.
\end{lemma}

\begin{proof}
The proof will be done by induction on $m$. In the case $m=1$ we have that
$e_j\in F(\{e_i\} , X_1)=F(\{e_i\},a_{11},\dots,a_{1 n-1})$ implying that 
$$e_i\in F(\{e_j\} , \overline{a}_{11},\dots,\overline{a}_{1 n-1})=F(\{e_j\} , \overline{X}_1),$$
by 1. of Lemma \ref{lema1}. Thus $(\overline{X}_1)$ is a connection from $e_j$ to $e_i$.

\medskip

Admit now that the assertion holds for any connection with $m\geq 1$ elements,  and let us show this assertion also holds for any connection $$(X_1,X_2,\dots,X_m,X_{m+1})$$ 
with $m+1$ ($(n-1)$-tuples) elements. So, consider a connection \newline $(X_1,X_2,\dots,X_m,X_{m+1})$ from $e_i$ to $e_j$. Let us begin by setting
$$U:=F( F (\dots (F(F(\{e_i\} , X_1),X_{2}),\dots,X_{m-1}), X_{m}).$$ 
Applying 2. of Definition \ref{connection} we have that $e_j \in F(U , X_{m+1}).$ Then, by 2. of Lemma \ref{lema1},
$F(\{e_j\} , \overline{X}_{m+1})\cap U\neq\emptyset$. Admit that 
\begin{equation}\label{eqq1}
x\in F(\{e_j\} , \overline{X}_{m+1})\cap U\neq \emptyset.
\end{equation}
Since $x\in U$ we have that $(X_1, X_2, \dots, X_{m-1}, X_m)$ is
a connection from $e_i$ to $x$ with $m$ elements. Henceforth
$(\overline{X}_m,\overline{X}_{m-1},\dots,\overline{X}_2,\overline{X}_1)$
connects $x$ to  $e_i$. From here,  and by equation (\ref{eqq1}), we
obtain
$$e_i \in F(F(\dots (F(F(\{e_j\} , \overline{X}_{m+1}), \overline{X}_m),\dots , \overline{X}_{2}),\overline{X}_{1}),$$
which means that
$$(\overline{X}_{m+1},\overline{X}_{m},\dots,\overline{X}_2,\overline{X}_1)$$
connects  $e_j$ to $e_i$.
\end{proof}

\begin{proposition}\label{equi}
The relation $\sim$ in ${\mathcal B}$, defined by $e_i\sim e_j$ if and only if $e_i$ is connected to $e_j$, is an equivalence relation.
\end{proposition}

\begin{proof}
The relation $\sim$ is clearly reflexive (see (i) of  Definition \ref{connection}) and symmetric (see Lemma \ref{lema3}). Hence let us verify its transitivity.

Admit that $e_i, e_j, e_k \in {\mathcal B}$ are pairwise different such that $e_i \sim e_j$ and $e_j \sim e_k$ (the cases in which two among those elements are equal are trivial). Then there are  connections $(X_1,\dots,X_m)$ and $(Y_1,\dots,Y_p)$ from $e_i$ to $e_j$ and from $e_j$ to $e_k$, respectively. Therefore, $(X_1,\dots,X_m,Y_1,\dots,Y_p)$ is a connection from $e_i$ to $e_k$ showing the transitivity of $\sim$, and the result is proved.
\end{proof}

\smallskip

Henceforth, by the above defined equivalence relation, we introduce the quotient set
$${\mathcal B}/ \sim := \{[e_i] : e_i \in {\mathcal B}\},$$
 where $[e_i]$ stands for the set of elements in ${\mathcal B}$ which are connected to $e_i$.

 \medskip

For each  $[e_i]  \in {\mathcal B}/ \sim$ we may introduce the linear subspace 
$$V_{[e_i]}:= \bigoplus_{e_j \in [e_i] } {\mathbb F} e_j,$$
allowing us to write
\begin{equation}\label{elp3}
\mathbb V = \bigoplus\limits_{[e_i]\in {\mathcal B}/ \sim} V_{[e_i]}.
\end{equation}
Next we show that this is a decomposition of $\mathbb V$ in pairwise $f$-orthogonal subspaces.


\begin{lemma}\label{elp1}
For any $[e_i],[e_j] \in {\mathcal B}/ \sim$ with $[e_i] \neq [e_j]$, we have that 
\begin{equation}
f(\mathbb V,\dots,V_{[e_i]}^{\left(k_{1}\right)},\dots,V_{[e_j] }^{\left(k_{2}\right)},\dots,\mathbb V)=0, \label{f-orth_dec}
\end{equation}
for all $k_{1},k_{2}\in \{1,\dots,n\},\ k_{1} \neq k_{2}$.
\end{lemma}

\begin{proof}
In order to prove (\ref{f-orth_dec}) it is sufficient to show that
\begin{equation*}
f(\xi_{\sigma (1)},\dots,V_{[e_i]}^{\left(k_{1}\right)},\dots,V_{[e_j] }^{\left(k_{2}\right)},\dots,\xi_{\sigma (n-2)})=0,
\end{equation*}
for any permutation $\sigma \in {\mathbb S}_{n-2}$, $ \xi_{1},\dots,\xi_{n-2} \in {\mathcal B}$. 
Admit the opposite assertion. Then there are $e_k \in [e_i]$, $e_p \in [e_j] $ and $v\in \mathbb V^{*}$ such that 
\begin{equation}
v=f(\xi_{\sigma (1)},\dots,e_{k}^{\left(k_{1}\right)},\dots,e_{p }^{\left(k_{2}\right)},\dots,\xi_{\sigma (n-2)}), \label{aux}
\end{equation}
for some $\sigma \in {\mathbb S}_{n-2}$. By definition of $F$, from (\ref{aux}) we may deduce two facts:
$$\text{ (i) } v\in F(\{e_{k}\},e_{p},\xi_{1},\dots,\xi_{n-2}),$$
$$\text{ (ii) }  v\in F(\{e_{p}\},e_{k},\xi_{1},\dots,\xi_{n-2}).$$
From (ii) and 1. of Lemma \ref{lema1}, we have 
$$\text{ (iii) }  e_{p}\in F(\{v\},\ov{e}_{k},\ov{\xi}_{1},\dots,\ov{\xi}_{n-2})$$
From (i) and (iii), we observe that $\left(X_{1},X_{2}\right)$, where
$$ X_{1}=\left(e_{p},\xi_{1},\dots,\xi_{n-2}\right) \text{ and }X_{2}=\left(\ov{e}_{k},\ov{\xi}_{1},\dots,\ov{\xi}_{n-2}\right),$$
is a connection from $e_{k}$ to $e_{p}$. Thus,  $[e_i]=[e_k]=[e_p]=[e_j]$, causing a contradiction.
\end{proof}

As a consequence of Lemma \ref{elp1} and equation (\ref{elp3}) we have.

\begin{proposition}\label{meta}
Given $\mathbb  V$ and $f$ as initially defined, $\mathbb V$ decomposes as  the $f$-orthogonal direct sum of linear subspaces $$\mathbb V = \bigoplus\limits_{[e_i]\in {\mathcal B}/ \sim}V_{[e_i]}.$$
\end{proposition}

\smallskip

The family of linear subspaces of $\mathbb V$ formed by all of the $V_{[e_i]}$, $[e_i]\in {\mathcal B}/ \sim$, 
which gives rise to the decomposition in Proposition \ref{meta},  is not  good enough for our purposes. So we need to introduce a new equivalence relation on this family,  as follows.

\medskip

We begin by  observing that the above mentioned decomposition of $\mathbb V$ 
allows us to consider, for each $V_{[e_i]}$,  the projection map $$\Pi_{V_{[e_i]}}: \mathbb V \to V_{[e_i]}.$$

Also, let us  consider these family of nonzero  linear subspaces of $\mathbb V$,
\begin{equation*}
\hbox{${\mathcal{F}}:=\{V_{[e_i]}:[e_i]\in {\mathcal B}/ \sim  \}.$}
\end{equation*}%

\begin{definition}\label{ane3}\rm

 We  will say that $V_{[e_i]} \approx V_{[e_j]}$ if and only if either $V_{[e_i]} = V_{[e_j]}$ or there exists a  subset
\begin{equation*}
\{[\xi_1],[\xi_2], \ldots,[\xi_m]\} \subset {\mathcal B}/ \sim,
\end{equation*}
 such that

\begin{itemize}
\item[(i)] $[\xi_1]=[e_i]$ and $[\xi_m]=[e_j].$

\item[(ii)] 
{\scriptsize 
$$\hspace{-0.75cm} \sum\limits_{1 \leq k_1<k_2\leq n}  \left[\Pi_{V_{[\xi_1]}}(f(\mathbb V, \ldots, V_{[\xi_2]}^{(k_1)}, \ldots, V_{[\xi_2]}^{(k_2)}, \ldots, \mathbb V)) + \Pi_{V_{[\xi_2]}}(f(\mathbb V, \ldots, V_{[\xi_1]}^{(k_1)}, \ldots,  V_{[\xi_1]}^{(k_2)},  \ldots, \mathbb V)) \right]\neq 0.$$
$$\hspace{-0.75cm}  \sum\limits_{1 \leq k_1<k_2\leq n}  \left[\Pi_{V_{[\xi_2]}}(f(\mathbb V, \ldots, V_{[\xi_3]}^{(k_1)}, \ldots, V_{[\xi_3]}^{(k_2)}, \ldots, \mathbb V)) + \Pi_{V_{[\xi_3]}}(f(\mathbb V, \ldots, V_{[\xi_2]}^{(k_1)}, \ldots,  V_{[\xi_2]}^{(k_2)},  \ldots, \mathbb V)) \right]\neq 0.$$
$\vdots$\\
$$\hspace{-0.75cm}  \sum\limits_{1 \leq k_1<k_2\leq n}  \left[\Pi_{V_{[\xi_{m-1}]}}(f(\mathbb V, \ldots, V_{[\xi_m]}^{(k_1)}, \ldots, V_{[\xi_m]}^{(k_2)}, \ldots, \mathbb V)) + \Pi_{V_{[\xi_m]}}(f(\mathbb V, \ldots, V_{[\xi_{m-1}]}^{(k_1)}, \ldots,  V_{[\xi_{m-1}]}^{(k_2)},  \ldots, \mathbb V)) \right]\neq 0.$$}

\end{itemize}
\end{definition}

\medskip

Clearly $\approx$ is an equivalence relation on ${\mathcal{F}}$ and so we
can introduce the quotient set
\begin{equation*}
{\mathcal{F}} / \approx :=\{ [V_{[e_i]}]: V_{[e_i]}
\in {\mathcal{F}}\}.
\end{equation*}
For each $[V_{[e_i]}] \in {\mathcal{F}} / \approx$,  we denote by $%
\overbrace{V_{[e_i]}}$ the linear subspace of $\mathbb V$
\begin{equation*}
\overbrace{V_{[e_i]}}:= \bigoplus\limits_{V_{[e_j]} \in [V_{[e_i]}]}   V_{[e_j]}.
\end{equation*}

By equation (\ref{elp3}) and  the definition of $\approx$, we clearly have

\begin{equation}\label{hel1}
\mathbb V=\bigoplus\limits_{[V_{[e_i]}] \in {\mathcal{F}} / \approx} \overbrace{V_{[e_i]}}.
\end{equation}
Also, we can assert by Lemma \ref{elp1} that

$$f(\mathbb V, \ldots, {{\overbrace{V_{[e_i]}}}}^{(k_1)}, \ldots,  {{\overbrace{V_{[e_j]}}}}^{(k_2)} \ldots, \mathbb V)=0$$
when $[V_{[e_i]}] \neq [V_{[e_j
]}]$ in ${\mathcal{F}} / \approx$, for all $k_1,k_2\in \left\{1,\dots,n \right\}$, $k_1 \neq k_2$.




\begin{proposition}\label{lema_submodulo}
 For any $[V_{[e_i]}] \in {\mathcal{F}} / \approx$,  $\overbrace{V_{[e_i]}}$ is a strongly $f$-invariant linear subspace  of $\mathbb V$.
\end{proposition}

\begin{proof}
We begin by proving that
\begin{equation}\label{panzer5}
 f(\mathbb V, \ldots, {\overbrace{V_{[e_i]}}}^{(k_1)}, \ldots, {\overbrace{V_{[e_i]}}}^{(k_2)}, \ldots, \mathbb  V) \subset \overbrace{V_{[e_i]}}.
 \end{equation}
 Indeed, in case some
$0 \neq w \in f(\mathbb  V, \ldots, {\overbrace{V_{[e_i]}}}^{(k_1)}, \ldots, {\overbrace{V_{[e_i]}}}^{(k_2)}, \ldots, \mathbb V)$, 
decomposition (\ref{hel1}) allows us to write $$w=w_1+w_2+ \dots +w_m$$  for some  $0\neq w_j \in \overbrace{V_{[\xi_j]}}$ for $j=1, \ldots ,m$ and $\xi_j \in \mathcal B$. 
Observe now that Lemma \ref{elp1} gives us that  there  exist nonzero 
$x,y \in V_{[e_k]}$ with $V_{[e_k]} \subset \overbrace{V_{[e_i]}}$ and $z_1, \ldots z_{n-2} \in \mathbb V,$ 
such that
\begin{equation}\label{boli}
0\neq w=f(z_1, \ldots, x^{(k_1)},\ldots, y^{(k_2)}, \ldots, z_{n-2})
\end{equation}
Let us   consider  $0 \neq w_1 \in \overbrace{V_{[\xi_1]}}$, being so $w_1 \in V_{[e_r]}$
for some $V_{[e_r]} \subset \overbrace{V_{[\xi_1]}}$. By equation (\ref{boli}) we have
 $$\Pi_{V_{[e_r]}}(f(z_1, \ldots, x^{(k_1)},\ldots, y^{(k_2)}, \ldots, z_{n-2}))=w_1 \neq 0.$$ 
 That is
$$\Pi_{V_{[e_r]}}(f(\mathbb V, \ldots, V_{[e_k]}^{(k_1)},\ldots, V_{[e_k]}^{(k_2)}, \ldots, \mathbb V )) \neq 0$$ 
and we get that the set $\{[e_k],[e_r]\}$ gives us $V_{[e_k]} \approx V_{[e_r] }$. Hence
$$V_{[e_i]} \approx V_{[e_k]} \approx V_{[e_r] } \approx
 V_{[\xi_1] }$$ and we conclude
$V_{[\xi_1] } \subset \overbrace{V_{[e_i]}}$. From here  $w_1 \in \overbrace{V_{[e_i]}}$. In a similar way we get that any $w_j \in \overbrace{V_{[e_i]}}$ for $j=2,...,m$ and so  $w  \in \overbrace{V_{[e_i]}}$. Consequently, the inclusion (\ref{panzer5}) holds,  as desired.

\medskip

Finally,  by decomposition (\ref{hel1}), 
Lemma \ref{elp1} and equation (\ref{panzer5}), we have the following inclusion
$$\sum\limits_{j=1}^{n}f(\mathbb V, \ldots, {\overbrace{V_{[e_i]}}}^{(j)}, \ldots, \mathbb V) \subset \overbrace{V_{[e_i]}},$$
and thus $f(\mathbb V, \ldots, {\overbrace{V_{[e_i]}}}^{(j)}, \ldots, \mathbb V) \subset \overbrace{V_{[e_i]}}$ for all $j\in\{1,\dots,n\}.$
\end{proof}

\begin{theorem}\label{theo1}
Let $\mathbb V$ be a linear space  equipped with an $n$-linear map \newline $f: \mathbb V \times \ldots \times \mathbb V \to \mathbb V$.  For any basis  ${\mathcal B} =\{e_i: i \in I\}$ of $\mathbb V$ we have that $\mathbb V$ decomposes as the $f$-orthogonal direct sum of strongly $f$-invariant linear subspaces
$$\mathbb V=\bigoplus\limits_{[V_{[e_i]}] \in {\mathcal{F}} / \approx} \overbrace{V_{[e_i]}}.$$
\end{theorem}

\begin{proof}
Consider the decomposition, as direct sum of linear subspaces $$\mathbb V=\bigoplus\limits_{[V_{[e_i]}] \in {\mathcal{F}} / \approx} \overbrace{V_{[e_i]}},$$ given by equation (\ref{hel1}). Now  Lemma \ref{elp1} shows that this decomposition is $f$-orthogonal and Proposition \ref{lema_submodulo}   that all of the linear subspaces  $\overbrace{V_{[e_i]}}$ are strongly $f$-invariant.
\end{proof}

\section{On the relation among  the decompositions given by different choices  of bases}

Observe that the decomposition of $\mathbb V$  as an $f$-orthogonal direct sum of strongly $f$-invariant linear subspaces given by Theorem \ref{theo1} is related with the initial choice of the basis. Indeed, as it was exemplified in \cite{Yo4}, for $n=2$, two different bases of $\mathbb V$
may lead to two different of those decompositions of $\mathbb V$.  The same happens in the $n$-ary case, with $n>2$, as shown in the following example.

Let $\mathbb V$ be the  ${\mathbb R}$-linear space $\mathbb V:={\mathbb R}^4$ equipped with  the $n$-linear map $f:{\mathbb R}^4 \times \dots \times {\mathbb R}^4 \to {\mathbb R}^4$ defined as
$$f(\overline{x}_1,\dots,\overline{x}_n)=(x_{11}x_{21},x_{11}x_{21}, 0,0),$$
where
$$\overline{x}_i=(x_{i1},\dots,x_{i4})$$
for each $i\in \{1,\dots,n\}.$

Let us consider the following two bases of ${\mathbb R}^4$:
$$ {\mathcal B}:=\{e_1,\dots,e_4\}, $$
that is, the canonical basis, and 
$${\mathcal B}^{\prime}:=\{(1,0,1,0), (1,0,-1,0),e_2, e_4\}.$$

Then it is possible to observe that the decomposition of $\mathbb V={\mathbb R}^4$, given in Theorem \ref{theo1} with respect to the  basis ${\mathcal B}$ is given by
$${\mathbb R}^4=({\mathbb R}e_1 \oplus {\mathbb R}e_2) \bigoplus ({\mathbb R}e_3) \bigoplus ({\mathbb R}e_4).$$

However, the same kind of decomposition with respect to ${\mathcal B}^{\prime}$ is given by 
$${\mathbb R}^4=({\mathbb R}(1,0,1,0) \oplus {\mathbb R}(1,0,-1,0) \oplus {\mathbb R} e_2)\bigoplus (  {\mathbb R}e_4).$$

Thus, it will be an interesting task to find a sufficient condition for two different decompositions of a linear space $\mathbb V$, induced by an $n$-linear map $f$ and with respect to two  different bases of $\mathbb V$, being isomorphic.  The following notion will help us in this purpose. 

\begin{definition}\rm
Let $\mathbb V$ be a linear space  equipped with an $n$-linear map $f: \mathbb V \times \dots \times \mathbb V \to \mathbb V$ and consider 
$$ \Gamma:=\ \mathbb V=\bigoplus\limits_{i \in I} V_i \text{ and } \Gamma^{\prime}:=\ \mathbb V=\bigoplus\limits_{j \in J} W_j$$
two decompositions of $\mathbb V$ as an $f$-orthogonal direct sum of strongly $f$-invariant linear subspaces. It is said  that $\Gamma$ and $\Gamma^{\prime}$ are {\it isomorphic} if there exists a linear  isomorphism $g:\mathbb V \to \mathbb V$ satisfying 
$$f(g(v_1),\dots g(v_n))=g(f(v_1,\dots , v_n))$$ 
for any $v_1,\dots , v_n \in \mathbb V$, and a bijection $\sigma: I \to J$ such that 
$$g(V_i)=W_{\sigma(i)}$$ 
for any $i \in I$.
\end{definition}

\begin{lemma}\label{genesis}
Let $\mathbb V$ be a linear space  equipped with an $n$-linear map \newline $f: \mathbb V \times \dots \times \mathbb V \to \mathbb V$ and consider  ${\mathcal B}=\{e_i: i\in I\}$ a  fixed  basis of $\mathbb V$.  Let also  $g:\mathbb V \to \mathbb V$ be a linear isomorphism satisfying
$$f\left(g\left(\xi_1\right),\dots,g\left(\xi_{n}\right)\right)=g\left(f\left(\xi_1,\dots,\xi_{n}\right)\right)$$ 
for any $\xi_i \in \mathcal B$. Then for any $U \in {\mathcal P}( \mathbb V^{*})$ and $\xi_k \in \mathcal B,\  k \in I$, the following assertions hold:

\begin{itemize}
\item[(i)] $g\left(F\left(U,\xi_1,\dots, \xi_{n-1}\right)\right)
=F\left(g(U),g\left( \xi_1 \right),\dots,g\left( \xi_{n-1}\right)\right)$,

\item[(ii)]  $g\left(F\left(U,\ov{\xi}_1,\dots,\ov{\xi}_{n-1}\right)\right)
=F\left(g(U),\ov{g\left({\xi}_1\right)},\dots,\ov{g\left({\xi}_{n-1}\right)}\right)$,

\end{itemize}
where $F$ is the mapping defined  by equation (\ref{gene}).
\end{lemma}

\begin{proof}

(i) We have 
$$\hspace{-1cm}  g\left(F\left(U, \xi_1,\dots,\xi_{n-1}\right)\right)
=\left( \bigcup_{  
\begin{array}{c}
k\in 	\{1,\dots,n\}\\  
\sigma \in {\mathbb S}_{n-1}
\end{array}
} \left\{g\left(f(\xi_{\sigma (1)},\dots, u^{(k)}, \ldots, \xi_{\sigma (n-1)})\right):u \in U\right\} \right)\setminus \{0\} $$
$$\hspace{-1cm}  =\left( \bigcup_{  
\begin{array}{c}
k\in 	\{1,\dots,n\}\\  
\sigma \in {\mathbb S}_{n-1}
\end{array}
} \left\{f \left(g\left(\xi_{\sigma (1)}\right),\dots, g(u)^{(k)}, \ldots, g\left(\xi_{\sigma (n-1)}\right)\right):u \in U\right\} \right)\setminus \{0\} $$
$$=F\left(g(U),g\left( \xi_1\right),\dots,g\left( \xi_{n-1}\right)\right).$$

(ii) In this case we have 

$$\hspace{-1cm} g\left(F\left(U,\ov{\xi}_1,\dots,\ov{\xi}_{n-1}\right)\right)$$
$$\hspace{-1cm}=\left( \bigcup_{  
\begin{array}{c}
k\in 	\{1,\dots,n\}\\  
\sigma \in {\mathbb S}_{n-1}
\end{array}
} \left\{u\in \mathbb V : f\left(\xi_{\sigma (1)},\dots, (g^{-1}(u))^{(k)}, \ldots, \xi_{\sigma (n-1)}\right) \in U\right\} \right)\setminus \{0\} $$
$$ \hspace{-1cm} =\left( \bigcup_{  
\begin{array}{c}
k\in 	\{1,\dots,n\}\\  
\sigma \in {\mathbb S}_{n-1}
\end{array}
} \left\{u\in \mathbb V : f \left(g\left( \xi_{\sigma (1)}\right),\dots, u^{(k)}, \ldots, g\left(\xi_{\sigma (n-1)}\right)\right) \in g(U)\right\} \right)\setminus \{0\} $$
$$=F\left(g(U),\ov{g\left({\xi}_1\right)},\dots,\ov{g\left({\xi}_{n-1}\right)}\right).$$

Observe that in both cases we took into account Remark \ref{Fsym}.
\end{proof}

\begin{proposition}\label{188}

Let $\mathbb V$ be a linear space  equipped with an $n$-linear map $f: \mathbb V \times \dots \times \mathbb V \to \mathbb V$ and consider  ${\mathcal B}=\{e_i: i\in I\}$ a  fixed   a basis of $\mathbb V$. Further, admit that $g:\mathbb V \to \mathbb V$ is a linear isomorphism satisfying
$$f\left(g\left(\xi_1\right),\dots,g\left(\xi_{n}\right)\right)=g\left(f\left(\xi_1,\dots,\xi_{n}\right)\right)$$ 
for any $\xi_i \in \mathcal B$. Then the decompositions
$$\Gamma:=\mathbb V=\bigoplus\limits_{[V_{[e_i]}] \in {\mathcal{F}} / \approx} \overbrace{V_{[e_i]}} \hbox{ and } \Gamma^{\prime}:=\mathbb V=\bigoplus\limits_{[V_{[g(e_i)]}] \in {\mathcal{F}^{\prime}} / \approx} \overbrace{V_{[g(e_i)]}},$$ corresponding to the choices of ${\mathcal B}$  and ${\mathcal B}^{\prime}:=\{g(e_i): i \in I\}$  respectively  in  Theorem \ref{theo1}, are isomorphic.
\end{proposition}

\begin{proof}
Firstly, let us observe that, according to the previous result, we may state that if $e_i$ is connected to some $e_j$, for some $i,j \in I$, $e_i,e_j \in \mathcal B$ through a connection $(X_1,X_2,\dots,X_m)$, where $X_{i}=\left(a_{i1},\dots,a_{in-1} \right)$ such that $a_{ik} \in  {\mathcal B}\dot\cup \overline{{\mathcal B}} $, $i \in \{1,\dots,m\},\ k\in \{1,\dots,n-1\}$, then $g(e_i)$ is connected to $g(e_j)$ through the connection $(g(X_1),g(X_2),...,g(X_n))$, where
$g(X_{i}):=\left(g(a_{i1}),\dots,g(a_{in-1}) \right)$ and $g(a_{ik}) \in   {\mathcal B^{\prime}} \cup \ov{{\mathcal B}^{\prime}}$, (where $g(\ov{e}_k):=\ov{g(e_k)}$).
\smallskip
Thus, it is possible to conclude that 
$$g(V_{[e_i]})=V_{[g(e_i)]}$$ 
for any $[e_i] \in {\mathcal B} / \sim$. Further, 
it is also clear that the mapping $\mu$ such that 
$$\mu(V_{[e_i]})=V_{[g(e_i)]}$$
defines a bijection between the families ${\mathcal{F}}:=\{V_{[e_i]}:[e_i]\in {\mathcal B}/ \sim  \}$ and
${\mathcal{F}}^{\prime}:=\{V_{[g(e_i)]}:[g(e_i)]\in {\mathcal B}^{\prime}/ \sim  \}$.

\smallskip

Now, from Lemma \ref{genesis}  we have
$$\hspace{-1cm}  g\left(\Pi_{V_{[e_i]}}(f(\mathbb V, \ldots, V_{[e_j]}^{(k_1)}, \ldots, V_{[e_j]}^{(k_2)}, \ldots, \mathbb V)\right)
=\Pi_{V_{[g(e_i)]}} 
\left( f(\mathbb V, \ldots, V_{[g(e_j)]}^{(k_1)}, \ldots, V_{[g(e_j)]}^{(k_2)}, \ldots, \mathbb V) \right)$$
for $i,j \in I$ and $k_1,k_2 \in \{1,\dots , n\}$, with $k_1<k_2$. This allows to deduce that 

\begin{equation}\label{gene1}
g(\overbrace{V_{[e_i]}})=\overbrace{V_{[g(e_i)]}}
\end{equation}
for any $[V_{[e_i]}] \in  {\mathcal{F}} / \approx$,  which induces a second bijection, $\sigma$, now between the families $ {\mathcal{F}} / \approx$ and
${\mathcal{F}^{\prime}} / \approx $ given by

\begin{equation}\label{gene2}
\sigma([V_{[e_i]}])=[V_{[g(e_i)]}].
\end{equation}

From equations (\ref{gene1}) and (\ref{gene2}) we conclude that the decompositions $\Gamma$ and $\Gamma^{\prime}$ are isomorphic.
\end{proof}

Being $f$ an $n$-linear map on $\mathbb V$, the following set 
 $${\rm O}_{f}(\mathbb V)=\{g \in {\rm GL}(\mathbb V): f(g(v_1),\dots,g(v_n))=g(f(v_1,\dots,v_n)) \hbox{ for any } v_1,\dots,v_n \in \mathbb V\},$$ 
(where $\rm{GL}(\mathbb V)$ denotes the group of all linear isomorphisms of $\mathbb V$), is known as the {\it orbit} of $\mathbb V$ (associated to $f$). We have that ${\rm O}_{f}(\mathbb V)$ is a subgroup of $\rm{GL}(\mathbb V)$. If we also denote by ${\mathfrak B}$ the set of all bases of $\mathbb V$ we get the action

\begin{equation}\label{act}
  {\rm O}_{f}(\mathbb V) \times {\mathfrak B} \to {\mathfrak B}
\end{equation}
   given by $(g, \{e_i\}_{i \in I})=\{g(e_i)\}_{i \in I}$. The previous result states that if two bases ${\mathcal B} $ and ${\mathcal B}^{\prime} $ of $\mathbb V$ belong to the same orbit under the action given by equation (\ref{act}), then they induce two isomorphic decompositions of $\mathbb V$. Finally, this can be stated as follows.

   \begin{corollary}
Let $\mathbb V$ be a linear space  equipped with an $n$-linear map \newline $f: \mathbb V \times \dots \times \mathbb V \to \mathbb V$ and fix two  bases  ${\mathcal B} =\{e_i: i \in I\}$ and ${\mathcal B}^{\prime} =\{u_i: i \in I\}$ of $\mathbb V$. Suppose there exists a bijection $\mu:I \to I$  such that  the linear isomorphism $g:\mathbb V \to \mathbb V$ determined by $g(e_i):=u_{\mu(i)}$ for any $i\in I$, satisfies 

$$f\left(g(v_1),\dots,u_{\mu(i)}^{(k_1)},\dots, u_{\mu(j)}^{(k_2)},\dots,g(v_{n-2})\right)=g(f(v_1,\dots,e_i^{(k_1)},\dots, e_j^{(k_2)},\dots,v_{n-2}))$$ 
for any $i,j \in I$, $k_1, k_2 \in \{1,\dots,n\}$, with $k_1<k_2$. Then the decompositions
$$\Gamma:=\mathbb V=\bigoplus\limits_{[V_{[e_i]}] \in {\mathcal{F}} / \approx} \overbrace{V_{[e_i]}} \hbox{ and } \Gamma^{\prime}:=\mathbb V=\bigoplus\limits_{[V_{[u_i]}] \in {\mathcal{F}^{\prime}} / \approx} \overbrace{V_{[u_i]}},$$ corresponding to the choices of ${\mathcal B}$  and ${\mathcal B}^{\prime}$,  respectively, in  Theorem \ref{theo1}, are isomorphic.
\end{corollary}

\section{A  characterization of the $f$-simplicity of the components }

Our aim in this section is to establish  a characterization theorem  on the $f$-simplicity of the linear subspaces $\overbrace{V_{[e_i]}}$,  which appear in the decomposition of $\mathbb V$ given in Theorem \ref{theo1}.

\medskip

Let us begin by recalling  several  concepts from the theory of algebras.

\medskip

Let $\mathbb{A}$ be an algebra equipped with an $n$-ary multiplication 
$[ .,\dots ,.]\ :\mathbb{A}\times \dots \times \mathbb{A} \to \mathbb{A}$ 
and  ${\mathcal{B}}$  a basis of $\mathbb{A}$. The basis ${\mathcal{B}}$  is said to be an \textit{$i$-division basis} if for any $e_i \in {\mathcal{B}}$ and $b_1,\dots,b_{n-1} \in \mathbb{A}$ such that 
$$[ b_1,\dots ,e_i^{(k)},\dots,b_{n-1} ] =w\neq 0$$ 
for some $k\in \{1,\dots,n\}$ we have that $e_i,b_1,\dots,b_{n-1}  \in {\mathcal{I}}(w)$, where ${\mathcal{I}}(w)$ denotes
the {\it ideal of $\mathbb A$ generated by $w$}.

\medskip
The above notion can be generalized to the case of a linear space  $\mathbb V$  equipped with an $n$-linear map $f: \mathbb V \times \dots \times \mathbb V \to \mathbb V$. We refer to the minimal strongly $f$-invariant subspace of $\mathbb V$ that contains $v$ as the  {\it strongly $f$-invariant subspace  of $\mathbb V$ generated by $v$}, and will be  denoted by ${\mathcal{I}}(v)$. Observe that the sum of two strongly $f$-invariant subspaces  of $\mathbb V$ is also a strongly $f$-invariant subspace, and that the whole  $\mathbb V$ is a trivial strongly $f$-invariant subspace.

\begin{definition}\rm
Let $\mathbb V$ be a linear space, ${\mathcal B}=\{e_i:i \in I\}$ a fixed basis of $\mathbb V$  and  $f: \mathbb V \times \dots \times \mathbb V \to \mathbb V$ an $n$-linear map. It is said that ${\mathcal B}$ is an {\it $i$-division basis} of $\mathbb V$ respect  to $f$,  if for any $e_i \in {\mathcal{B}%
}$ and $b_1,\dots,b_{n-1} \in \mathbb{V}$ such that 
$$f\left( b_1,\dots ,e_i^{(k)},\dots,b_{n-1} \right) =w\neq 0$$ 
for some $k\in \{1,\dots,n\}$ we have that $e_i,b_1,\dots,b_{n-1}  \in {\mathcal{I}}(w)$, where ${\mathcal{I}}(w)$ denotes
the strongly $f$-invariant subspace  of $\mathbb V$ generated by $w$.
\end{definition}

Let us return to the decomposition of the linear space $\mathbb V$, given an $n$-linear map $f: \mathbb V \times \dots \times \mathbb V \to \mathbb V$  and fixed a basis ${\mathcal B}$, $$\mathbb V=\bigoplus\limits_{[V_{[e_i]}] \in {\mathcal{F}} / \approx} \overbrace{V_{[e_i]}}$$
as deduced by Theorem \ref{theo1}.  
For any 
$\overbrace{V_{[e_i]}}$ we can restrict $f$ to the $n$-linear map 
$$f^{\prime}: \overbrace{V_{[e_i]}} \times \dots \times \overbrace{V_{[e_i]}} \to \overbrace{V_{[e_i]}}$$ and consider on $\overbrace{V_{[e_i]}}$ the basis ${\mathcal B}^{\prime}:= {\mathcal B} \cap \overbrace{V_{[e_i]}}$. Then we can assert:

\begin{theorem}\label{ane100}
  The linear space $\overbrace{V_{[e_i]}}$ is $f^{\prime}$-simple  if and only if
${\rm Ann}(f^{\prime})=0$ and  ${\mathcal B}^{\prime}$ is an $i$-division basis of $\overbrace{V_{[e_i]}}$ with respect to $f^{\prime}$.
\end{theorem}

\begin{proof}

Suppose that $\overbrace{V_{[e_i]}}$ is $f^{\prime}$-simple. 
Observe firstly that ${\rm Ann}(f^{\prime})$ is a strongly $f^{\prime}$-invariant subspace of $\overbrace{V_{[e_i]}}$, and thus ${\rm Ann}(f^{\prime})=0.$ Additionally, if we consider
some  $e_j \in {\mathcal{B}}^{\prime}$ and $b_1,\dots,b_{n-1} \in \overbrace{V_{[e_i]}}$ such that
$$f^{\prime}\left( b_1,\dots ,e_j^{(k)},\dots,b_{n-1} \right) =w\neq 0$$ 
for some $k\in \{1,\dots,n\}$, since $\overbrace{V_{[e_i]}}$ is  $f^{\prime}$-simple, we have
$${\mathcal{I}}(w)=\overbrace{V_{[e_i]}}$$ 
and so $e_j,b_1,\dots,b_{n-1} \in {\mathcal{I}}(w)$. Thus, the basis ${\mathcal B}^{\prime}$ is an $i$-division basis of $\overbrace{V_{[e_i]}}$ with respect to $f^{\prime}$.

\medskip

Conversely, let us suppose that ${\rm Ann}(f^{\prime})=0$ and that the set ${\mathcal B}^{\prime}$ is an $i$-division basis of $\overbrace{V_{[e_i]}}$ with respect to $f^{\prime}$. Consider  any nonzero strongly $f^{\prime}$-invariant linear subspace $W$   of $\overbrace{V_{[e_i]}}$ and take some nonzero $ w \in W$. Since ${\rm Ann}(f^{\prime})=0$, there are nonzero elements 
$$ \xi_1,\dots,\xi_{n-1} \in {\mathcal B}^{\prime}$$ 
such that
$$0 \neq f\left( \xi_1,\dots ,w^{(j)},\dots,\xi_{n-1} \right) \in W$$ 
for some $j\in \{1,\dots,n\}.$ 
Since ${\mathcal B}^{\prime}$ is an $i$-division basis, we get
\begin{equation}\label{ane1}
 \xi_k \in W,
 \end{equation}
 for all $k\in \{1,\dots,n-1\}$.

Let us now prove that $ V_{[\xi_k]} \subset W$ for each $k\in \{1,\dots,n-1\}$. To do so, we have to show  that for any
 $\nu_j \in [\xi_k]$ such that  $\nu_j \neq \xi_k$,
 we must conclude that $\nu_j \in W$. It is clear that  $\xi_k$ is connected to any $\nu_j\in [\xi_k]$, and thus there is a connection
$ (X_1,X_2,...,X_m)$,  where $X_{i}=\left(a_{i1},\dots,a_{in-1} \right)$ such that $a_{il} \in  {\mathcal B}\dot\cup \overline{{\mathcal B}} $, $i \in \{1,\dots,m\},\ l\in \{1,\dots,n-1\}$, from $\xi_k$ to $\nu_j$.
 
 Recall that we are dealing with an $f$-orthogonal and strongly $f$-invariant  decomposition of $\mathbb{V}$ (by  Theorem \ref{theo1}). Thus, we may claim that the elements $a_{il}$ satisfy
 \begin{equation}\label{ane2}
 a_{il} \in {\mathcal B}^{\prime} \cup \overline{ {\mathcal B}^{\prime} },
 \end{equation}
 and that the whole connection process from $\xi_k$ to $\nu_j$ can be  deduced in $\overbrace{V_{[e_i]}}$.


 \medskip
 We have that $$F(\{\xi_k\}, X_1)=F(\{\xi_k\}, a_{11},\dots,a_{1 n-1}) \neq \emptyset.$$ There are two cases to discuss. \\
 First case: $a_{1l} \in {\mathcal B}^{\prime}$, $l=1,\dots,n-1$ and so there exists 
 $$0\neq x=f\left( a_{11},\dots ,\xi_k^{(r)},\dots,a_{1 n-1} \right),$$
 for some $r\in \{1,\dots,n\}$. \\
 Second case: 
 $a_{1l} \in \overline{{\mathcal B}^{\prime}}$, $l=1,\dots,n-1$ and so there exists 
 $0\neq x \in \overbrace{V_{[e_i]}}$ such that 
 $$f\left( \overline{a}_{11},\dots ,x^{(r)},\dots,\overline{a}_{1 n-1} \right)= \xi_k,$$
 for some $r\in \{1,\dots,n\}$.
 
 Consider the first case. As a consequence of the inclusion (\ref{ane1}), we obtain  $  x \in W.$
 
 Consider now the second case. By the $i$-division property of the basis ${\mathcal B}^{\prime}$ and due to inclusion (\ref{ane1})  we conclude that  $x \in {\mathcal I}(\xi_k) \subset W$.  
   
So, in both cases we have shown that
\begin{equation}\label{fel1}
F(\{\xi_k\}, X_1) \subset  W.
\end{equation}

\medskip

By the connection definition, we have 
$$F(F(\{\xi_k\}, X_1), X_2)\neq \emptyset,$$ 
where  $F(\{\xi_k\}, a_1) \subset W$ as seen in (\ref{fel1}).

  Given an arbitrary $t \in F(F(\{\xi_k\}, X_1), X_2)$, as before, we have two cases to distinguish. In the first one
 $a_{2l} \in {\mathcal B}^{\prime}$, $l=1,\dots,n-1$ and so there exists $z \in F(\{\xi_k\}, X_1)$ such that
 $$0\neq z=f\left( a_{21},\dots ,\xi_k^{(r^{\prime})},\dots,a_{2 n-1} \right),$$
 for some $r^{\prime}\in \{1,\dots,n\}$.\\
 In the second one $a_{2l} \in \overline{{\mathcal B}^{\prime}}$,  and then  there exists $z \in F(\{\xi_k\}, X_1)$  such that $0\neq f(\overline{a}_{21},\dots ,t^{(r^{\prime})},\dots,\overline{a}_{2 n-1}) =z$ .

\medskip

  In the first case the inclusion (\ref{fel1}) shows that $t \in W.$
  In the second case  the $i$-division property of ${\mathcal B}^{\prime}$ gives us that
  $t \in {\mathcal I}(z) \subset W$. 
  
  In both cases, we have
  $$F(F(\{\xi_k\}, X_1), X_2)\subset W.$$

Iterating this argument on the connection (\ref{ane2}), we obtain that 
$$\nu_j \in  F( F (\dots (F(F(\{\xi_k\},X_1),X_{2}),\dots,X_{m-1}), X_{m}) \subset W$$ 
and so we can assert that
\begin{equation}\label{ane4}
V_{[\xi_k]} \subset W.
\end{equation}

\medskip

To finish the proof, we must prove that  all $V_{[\nu_j]}$ such that $V_{[\nu_j]}\approx   V_{[\xi_k]}$ verifies $V_{[\nu_j]} \subset W.$

Under the above assumption, there exists  a subset
\begin{equation}\label{metalli}
\{[\xi_k],[\nu_2],...,[\nu_j]\} \subset {\mathcal B}/ \sim
\end{equation}
 satisfying the conditions in Definition \ref{ane3}. From here, 
 
 $$\sum\limits_{1 \leq i<i^{\prime} \leq n}  \left[\Pi_{V_{[\xi_k]}}(f(\mathbb V, \scriptsize{\ldots}, V_{[\nu_2]}^{(i)}, \scriptsize{\ldots}, V_{[\nu_2]}^{(i^{\prime})}, \scriptsize{\ldots}, \mathbb V)) + \Pi_{V_{[\nu_2]}}(f(\mathbb V, \scriptsize{\ldots}, V_{[\xi_k]}^{(i)}, \scriptsize{\ldots},  V_{[\xi_k]}^{(i^{\prime})},  \scriptsize{\ldots}, \mathbb V)) \right]\neq 0.$$
 Therefore, there are $i,i^{\prime}\in \{1,\dots,n\}$ with $i<i^{\prime}$, such that
$$\Pi_{V_{[\nu_2]}}(f(\mathbb V, \ldots, V_{[\xi_k]}^{(i)}, \ldots,  V_{[\xi_k]}^{(i^{\prime})},  \ldots, \mathbb V))\neq 0$$
or 
 $$\Pi_{V_{[\xi_k]}}(f(\mathbb V, \ldots, V_{[\nu_2]}^{(i)}, \ldots, V_{[\nu_2]}^{(i^{\prime})}, \ldots, \mathbb V))\neq 0.$$
 \medskip

Consider the first case, in which $$\Pi_{V_{[\nu_2]}}(f(\mathbb V, \ldots, V_{[\xi_k]}^{(i)}, \ldots,  V_{[\xi_k]}^{(i^{\prime})},  \ldots, \mathbb V))\neq 0.$$ 
Then there exist $e_k^{\prime},  e_k^{\prime \prime } \in [\xi_k]$ and $b_1,\dots,b_{n-2}\in \mathbb V$ such that

$$0 \neq f(b_1,\dots,{e_k^{\prime}} ^{(i)}, \dots,  {e_k^{\prime \prime}} ^{(i^{\prime})},  \dots, b_{n-2})=x_2+c$$
where $0 \neq x_2 \in V_{[\nu_2]}$ and  $c \in \bigoplus\limits_{[\nu_j] \neq [\nu_2]} V_{[\nu_j]}$.

Since ${\rm Ann}(f^{\prime})=0$, and taking into account Lemma \ref{elp1}, there exist
$e_{21}^{\prime},\dots,e_{2 n-1}^{\prime} \in [\nu_2]$ such that

$$0\neq f(e_{21}^{\prime},\dots,{x_2} ^{(r)}, \dots,    e_{2 n-1}^{\prime}) =q$$
for some $r\in \{1,\dots,n\}$.  By Lemma \ref{elp1}  and (\ref{ane4}) we have that
$$ 0\neq  f(e_{21}^{\prime},\dots,{f(b_1,\dots,{e_k^{\prime}} ^{(i)}, \dots,  {e_k^{\prime \prime}} ^{(i^{\prime})},  \dots, b_{n-2})} ^{(r)}, \dots,    e_{2 n-1}^{\prime}) =$$

$$ f(e_{21}^{\prime},\dots,{(x_2+c)} ^{(r)}, \dots, e_{2 n-1}^{\prime}) =f(e_{21}^{\prime},\dots,{x_2} ^{(r)}, \dots, e_{2 n-1}^{\prime}) =q\in  W.$$ 
From here, by  the $i$-division property of ${\mathcal B}^{\prime}$  we conclude that 
$$e_{21}^{\prime},\dots,e_{2 n-1}^{\prime} \in {\mathcal I}(q) \subset W.$$

Concerning the second case, recall that we have 
$$\Pi_{V_{[\xi_k]}}(f(\mathbb V, \ldots, V_{[\nu_2]}^{(i)}, \ldots, V_{[\nu_2]}^{(i^{\prime})}, \ldots, \mathbb V))\neq 0.$$
Similarly to the first case, there exist
$e_2^{\prime},  e_2^{\prime \prime } \in [\nu_2]$ and $b_1,\dots,b_{n-2}\in \mathbb V$ such that
$$0 \neq f(b_1,\dots,{e_2^{\prime}} ^{(i)}, \dots,  {e_2^{\prime \prime}} ^{(i^{\prime})},  \dots, b_{n-2})=x_k+d$$
where $0 \neq x_k \in V_{[\xi_k]}$ and  $d \in \bigoplus\limits_{[\nu_j] \neq [\xi_k]} V_{[\nu_j]}$.
 Again, since ${\rm Ann}(f^{\prime})=0$, there exist
$e_{k1}^{\prime},\dots,e_{k n-1}^{\prime} \in [\xi_k]$ such that 
 
$$0\neq f(e_{k1}^{\prime},\dots,{x_k} ^{(r)}, \dots,    e_{k n-1}^{\prime}) =s$$
for some $r\in \{1,\dots,n\}$.

By Lemma \ref{elp1}  and inclusion (\ref{ane4}) we have that
$$ 0\neq f(e_{k1}^{\prime},\dots,{f(b_1,\dots,{e_2^{\prime}} ^{(i)}, \dots,  {e_2^{\prime \prime}} ^{(i^{\prime})},  \dots, b_{n-2})} ^{(r)}, \dots,    e_{k n-1}^{\prime}) =$$

$$ f(e_{k1}^{\prime},\dots,{(x_k+d)} ^{(r)}, \dots, e_{k n-1}^{\prime}) =f(e_{k1}^{\prime},\dots,{x_k} ^{(r)}, \dots, e_{k n-1}^{\prime}) =s\in  W.$$ 
From here, by  the $i$-division property of ${\mathcal B}^{\prime}$  we conclude that 
$$e_{21}^{\prime},\dots,e_{2 n-1}^{\prime} \in {\mathcal I}(q) \subset W.$$
Applying the $i$-division property of ${\mathcal B}^{\prime}$ this leads to 

$$f(b_1,\dots,{e_2^{\prime}} ^{(i)}, \dots,  {e_2^{\prime \prime}} ^{(i^{\prime})},  \dots, b_{n-2})\in  {\mathcal I}(s) \subset W.$$
A second application of the $i$-division property of ${\mathcal B}^{\prime}$ allows us to write $e_2^{\prime} \in W$.

\medskip

At this point, we have shown in both cases that  there are elements in $[\nu_2]$ belonging to  $W$. Hence by using the same previous argument  as  done with $\xi_k$, (see inclusions (\ref{ane1}) and (\ref{ane4})), we get that $$V_{[\nu_2]} \subset W.$$

\medskip

It is clear that this reasoning can be repeated for all other elements of the set (\ref{metalli}). Henceforth
$$V_{[\nu_j]} \subset W$$
and consequently, since $$\overbrace{V_{[e_i]}}=\overbrace{V_{[\xi_k]}}:= \bigoplus\limits_{V_{[e_j]} \in [V_{[\xi_k]}]}   V_{[e_j]}$$ 
we proved that
$$\overbrace{V_{[e_i]}}=W,$$
that is  $\overbrace{V_{[e_i]}}$ is $f$-simple.
\end{proof}

\begin{remark}\rm
The above result can be restated as follows. 

{\it  The linear space $\overbrace{V_{[e_i]}}$ is $f^{\prime}$-simple  if and only if
${\rm Ann}(f^{\prime})=0$ and  every non-zero element in $\overbrace{V_{[e_i]}}$  is an $i$-division element with respect to $f^{\prime}$.}

\end{remark}

\section{Application to the structure theory of arbitrary $n$-ary algebras}

In this section we will apply the results obtained in the previous sections   to the  structure theory of arbitrary $n$-ary algebras.

\medskip

We  will denote by ${\mathfrak A}$   an arbitrary
$n$-ary algebra in the sense that there are no restrictions on the
 dimension of the algebra nor on the base field ${\mathbb F}$, and that
 no specific identity on the product ($n$-Lie (Filippov) \cite{Fil}, $n$-ary Jordan \cite{kps},  $n$-ary Malcev \cite{Pojidaev}, etc.) is supposed. 
 That is, ${\mathfrak A}$
 is just a linear space over ${\mathbb F}$ endowed with a $n$-linear
 map $$[ \cdot, \ldots, \cdot ]  :{\mathfrak A} \times \ldots \times  {\mathfrak A} \to {\mathfrak A}$$
 $$\hspace{1cm}(x_1, \ldots, x_n) \mapsto [x_1, \ldots, x_n]$$
 called {\it the product} of ${\mathfrak A}$.

 \medskip

 We recall that given an  $n$-ary algebra $({\mathfrak A}, [\cdot, \ldots, \cdot ])$, a {\it subalgebra} of ${\mathfrak A}$ is a linear subspace ${\mathfrak B}$ closed for the product. 
 That is, such that  $[{\mathfrak B}, \ldots, {\mathfrak B}] \subset {\mathfrak B}$. A linear subspace ${\mathfrak I}$ of ${\mathfrak A}$ is called an {\it ideal} of ${\mathfrak A}$  
 if $[{\mathfrak A}, \ldots,  {\mathfrak I}^{(r)}, \ldots, {\mathfrak A}]  \subset {\mathfrak I},$ for all $ r\in\{1,\dots,n\}$. An $n$-ary algebra ${\mathfrak A}$ is said to be {\it simple} if its product is nonzero and its only ideals are $\{0\}$ and ${\mathfrak A}$. 
 We finally recall that the {\it annihilator} of the algebra $({\mathfrak A}, [.,\dots,.])$ is defined as the linear subspace
 $${\rm Ann}({\mathfrak A})=\{x \in {\mathfrak A}:  [{\mathfrak A}, \ldots, x^{(k)}, \ldots,  {\mathfrak A}] =0, \mbox{ for all $k\in\{1,\dots,n\}$ }\}.$$

 \medskip

 If we  fix any basis ${\mathcal B}=\{e_i\}_{i \in I}$ of ${\mathfrak A}$, and denote the product $[.,\dots,.]$ of ${\mathfrak A}$ as $f$,  Theorem \ref{theo1} applies to get that
 ${\mathfrak A}$ decomposes as the $f$-orthogonal direct sum of strongly $f$-invariant linear subspaces
$${\mathfrak A}=\bigoplus\limits_{[{\mathfrak A}_{[e_i]}] \in {\mathcal{F}} / \approx} \overbrace{{\mathfrak A}_{[e_i]}}.$$

\medskip

Now observe that the $f$-orthogonality of the linear subspaces means that, when $[{\mathfrak A}_{[e_i]}] \neq [{\mathfrak A}_{[e_j]}]$, we have 
$$[ \mathfrak A, \ldots, \overbrace{{\mathfrak A}_{[e_i]}}^{(k_1)}, \ldots, \overbrace{{\mathfrak A}_{[e_j]}}^{(k_2)}, \ldots, \mathfrak A]=0,$$ 
for all $k_1,k_2\in \left\{1,\dots,n \right\}$, $k_1 \neq k_2$, and that the strongly $f$-invariance of a linear subspace $\overbrace{{\mathfrak A}_{[e_i]}}$ means that $\overbrace{{\mathfrak A}_{[e_i]}}$ is actually an ideal of ${\mathfrak A}$. From here, we can state:

\begin{theorem}
Let $({\mathfrak A}, [\cdot, \ldots, \cdot ])$ be an arbitrary  algebra. Then for any basis  ${\mathcal B} =\{e_i: i \in I\}$ of ${\mathfrak A}$
one has the decomposition
$${\mathfrak A}=\bigoplus\limits_{[{\mathfrak A}_{[e_i]}] \in {\mathcal{F}} / \approx} \overbrace{{\mathfrak A}_{[e_i]}},$$
being any
$\overbrace{{\mathfrak A}_{[e_i]}}$ an  ideal of $ {\mathfrak A}$. Furthermore, any pair of different ideals in this decomposition is $f$-orthogonal.
\end{theorem}

\smallskip

In the same context, if we restrict the product $[\cdot, \ldots, \cdot ]$ of ${\mathfrak A}$ to any ideal $\overbrace{{\mathfrak A}_{[e_i]}}$, we get the algebra $(\overbrace{{\mathfrak A}_{[e_i]}}, [\cdot, \ldots, \cdot ])$. Now, by observing that the $f'$-simplicity of $(\overbrace{{\mathfrak A}_{[e_i]}}, [\cdot, \ldots, \cdot ])$ is equivalent to the simplicity of $(\overbrace{{\mathfrak A}_{[e_i]}}, [\cdot, \ldots, \cdot ])$ as an algebra, and  that ${\rm Ann}(f^{\prime})={\rm Ann}(\overbrace{{\mathfrak A}_{[e_i]}})$, Theorem \ref{ane100}
allows us to assert the following.

\begin{theorem}
  The ideal  $(\overbrace{{\mathfrak A}_{[e_i]}}, [\cdot, \ldots, \cdot ])$ is  simple  if and only if
${\rm Ann}(\overbrace{{\mathfrak A}_{[e_i]}})=0$ and  ${\mathcal B}^{\prime}:={\mathcal B} \cap \overbrace{{\mathfrak A}_{[e_i]}}$ is an $i$-division basis of $\overbrace{{\mathfrak A}_{[e_i]}}$.
\end{theorem}

\end{document}